\documentclass[11pt,a4paper,french,english]{article}
\usepackage[top=3.5cm, bottom=3.5cm, left=3cm, right=3cm]{geometry}
\usepackage[T1]{fontenc}
\usepackage{lmodern}
\usepackage[latin1]{inputenc}
\usepackage{csquotes}
\usepackage{amsmath}
\usepackage{amsthm}
\usepackage{amsfonts}
\usepackage{amssymb}
\usepackage{graphics}
\usepackage{float}
\usepackage{breqn}

\usepackage[hang,flushmargin]{footmisc} 

\usepackage{amsmath,amsthm,amssymb,graphicx, multicol, array}
\usepackage{enumerate}
\usepackage{enumitem}
\usepackage{hyperref}
\usepackage{setspace}
\usepackage{xcolor}

\usepackage{tabto}

\numberwithin{equation}{section}

\usepackage{tikz}
\usepackage{pgfplots}

\usetikzlibrary{matrix}

\usepackage[mathscr]{euscript}
\usepackage{mathrsfs}

\newtheorem{thm}{Theorem}[section]
\newtheorem{lem}{Lemma}[section]
\newtheorem{theorem}{Theorem}[section]

\newtheorem{conj}{Conjecture}[section]

\newtheorem{cor}{Corollary}[section]
\newtheorem{dfn}{Definition}[section]

\newcommand{\N}{\mathbb{N}}
\newcommand{\Z}{\mathbb{Z}}
\newcommand{\Q}{\mathbb{Q}}
\newcommand{\R}{\mathbb{R}}
\newcommand{\C}{\mathbb{C}}
\newcommand{\F}{\mathbb{F}}

\newcommand{\I}{\mathbb{I}}

\usepackage{fancyhdr}
\title{ Linear Independence of Some Irrational Numbers}
\date{}
\author{N. A. Carella}

\begin{document}
\thispagestyle{empty}
\date{}

\maketitle
\textbf{\textit{Abstract}:} This note presents an analytic technique for proving the linear independence of certain small subsets of real numbers over the rational numbers. The applications of this test produce simple linear independence proofs for the subsets of triples $\{1, e, \pi\}$, $\{1, e, \pi^{-1}\}$, and $\{1, \pi^r, \pi^s\}$, where $1\leq r<s $ are fixed integers. \let\thefootnote\relax\footnote{ \today \date{} \\
\textit{AMS MSC}:Primary 11J72; Secondary 11K06 \\
\textit{Keywords}: Irrational number; Irrationality criteria; Linear independence; Uniform distribution.}


\section{Introduction} \label{sect1010}\hypertarget{sect1010}
The algebra and number theory literature has many elementary techniques used to verify the linear independence of small finite subsets of algebraic numbers $\{\alpha_1, \alpha_2, \ldots, \alpha_d\}\subset \overline{\Q}$ over the rational numbers $\Q$. A few examples of these algebraic subsets are
 \begin{enumerate} 
\item $\{1, \sqrt{2},\sqrt[3]{2},\sqrt[4]{2}\}\subset \overline{\Q}$,
\item $\{1, \sqrt{2},\sqrt{3},\sqrt{5}\}\subset \overline{\Q}$,
\item $\{1, \alpha,\alpha^2,\ldots,\alpha^d\}\subset \overline{\R}$, where $f(\alpha)=0$,
\item $\{1, \omega,\omega^2,\ldots \omega^{\varphi(n)}\}\subset \overline{\R}$, where $\omega^n=1$.
\end{enumerate}
But, these techniques are intrinsically algebraic, and do not seem to be applicable to small subsets of nonalgebraic numbers $\{\alpha_1, \alpha_2, \ldots, \alpha_d\}\subset \overline{\R}$. This note presents an analytic technique for proving the linear independence of certain small subsets of nonalgebraic numbers $\{\alpha_1, \alpha_2, \alpha_3\}\subset \overline{\R}$ over the rational numbers $\Q$. The applications of this test produce simple proofs for the followings subsets of triples.

\begin{thm} \label{thm7575.26}  The real numbers $1$, $e$ and $\pi$ are linearly independent over the algebraic integers.
\end{thm}  

\begin{thm} \label{thm8575.26}  The real numbers $1$, $e$ and $\pi^{-1}$ are linearly independent over the algebraic integers.
\end{thm}  

\begin{thm} \label{thm8008.26}  For any pair of integers $1\leq r< s$, the real numbers $1$, $\pi^r$ and $\pi^{s}$ are linearly independent over the algebraic integers.
\end{thm}
The proofs are presented in \hyperlink{s7575}{Section} \ref{s7575}, \hyperlink{s8575}{Section} \ref{s8575}, and \hyperlink{s8008}{Section} \ref{s8008} respectively. The remaining sections are of similar topics or independent interest.\\

\section{Results for Sum and Product of $e$ and $\pi$} \label{S5577-I}\hypertarget{S5577-I}
The real numbers $e$ and $\pi$ are transcendental numbers and it is immediate that at least one, the trace $Tr(\alpha)= e+\pi$ or the norm $N(\alpha)= e\pi$ of the polynomial 
\begin{equation}
	f_0(x)=(x- e)(x-\pi)
\end{equation} 
is a transcendental number. Similarly, at least one, the trace $Tr(\alpha)= e^2\pi+e^{-1}$ or the norm $N(\alpha)= e\pi$ of the polynomial 
\begin{equation}
	f_1(x)=(x- e^2\pi)(x-e^{-1})
\end{equation}
is a transcendental number.\\

An elementary proof that all these traces and norms are transcendental numbers is given here. 

\begin{theorem}\label{thm5577.332V}\hypertarget{thm5577.332V} The real numbers
	\begin{equation}
		e\pi, \qquad e+ \pi, \qquad e^2\pi+ e^{-1}
	\end{equation}	
	are transcendental numbers. 
\end{theorem}
\begin{proof} The two systems of equations associated with the traces and norms of the polynomials $f_0(x)$ and $f_1(x)$ are treated separately. \\
	
	\textbf{Case I.} Suppose that 
	\begin{align}\label{eq5577.332V01}
		e\pi&=\alpha\in \mathbb{R}-\overline{\mathbb{Q}}\\[.2cm]
		e+\pi&=\frac{A}{B}\label{eq5577.332V01b},
	\end{align}
	where $\alpha>0$ is a transcendental number and $A, B>1$ are algebraic integers.  Substitute \eqref{eq5577.332V01} into \eqref{eq5577.332V01b} and solve for $\pi$ in terms of $\alpha$ and $A,B>1$. This step yields\\
	\begin{equation}\label{eq5577.332V03}
		\alpha \pi^2-\frac{A}{B}\pi+\alpha=0
	\end{equation}
	and 
	\begin{equation}\label{eq5577.332V05}
		\pi=\frac{\frac{A}{B}\pm\sqrt{\left (\frac{A}{B}\right )^2-4\alpha}}{2},
	\end{equation}
	where 
	\begin{equation}\label{eq5577.332V09}
		\left (\frac{A}{B}\right )^2-4\alpha>0.
	\end{equation}
	\textbf{Case II.} Suppose that
	\begin{align}\label{eq5577.332V11}
		e\pi&=\alpha\in \mathbb{R}-\overline{\mathbb{Q}}\\[.2cm]
		e^2\pi+ e^{-1}&=\frac{C}{D},\label{eq5577.332V11b}
	\end{align}
	where $\alpha>0$ is a transcendental number and $C, D>1$ are algebraic integers. Substitute \eqref{eq5577.332V11} into \eqref{eq5577.332V11b} and solve for $\pi$ in terms of $\alpha$ and $C,D>1$. This step yields\\
	
	\begin{equation}\label{eq5577.332V13}
		\pi^2-\frac{C}{D}\alpha \pi+ \alpha^{3}=0
	\end{equation}
	and 
	\begin{equation}\label{eq5577.332V15}
		\pi=\frac{\alpha\frac{C}{D}\pm\sqrt{\left (\alpha\frac{C}{D}\right )^2-4\alpha^3}}{2},
	\end{equation}
	where 
	\begin{equation}\label{eq5577.332V21}
		\left (\alpha\frac{C}{D}\right )^2-4\alpha^3>0.
	\end{equation}
	Merging \eqref{eq5577.332V05} and \eqref{eq5577.332V15} yield
	\begin{equation}\label{eq5577.332V23}
		\frac{\frac{A}{B}\pm\sqrt{\left (\frac{A}{B}\right )^2-4\alpha}}{2\alpha}=\frac{\alpha\frac{C}{D}\pm\sqrt{\left (\alpha\frac{C}{D}\right )^2-4\alpha^3}}{2}.
	\end{equation}
	Clearly, the minimal polynomial $f(t)\in \mathbb{Z}[t]$ of $\alpha$ has algebraic integer coefficients. This implies that the transcendental number $\alpha>0$ satisfies an algebraic polynomial equation $f(\alpha)=0$. Ergo, the hypothesis stated in \eqref{eq5577.332V01} and \eqref{eq5577.332V11} is false.
	
	Reversing the hypothesis also leads to two systems of equations associated with the traces and norms. These are treated separately. \\
	
	\textbf{Case III.} Suppose that 
	\begin{align}\label{eq5577.332V02}
		e\pi&=\frac{A}{B}\\[.2cm]
		e+\pi&=\alpha\in \mathbb{R}-\overline{\mathbb{Q}},\label{eq5577.332V02b}
	\end{align}
	where $\alpha>0$ is a transcendental number and $A, B>1$ are algebraic integers.  Substitute \eqref{eq5577.332V02} into \eqref{eq5577.332V02b} and solve for $\pi$ in terms of $\alpha$ and $A,B>1$. This step yields\\
	\begin{equation}\label{eq5577.332V04}
		\pi^2-\alpha\pi+\frac{A}{B}=0
	\end{equation}
	and 
	\begin{equation}\label{eq5577.332V06}
		\pi=\frac{\alpha \pm\sqrt{\alpha^2-4\frac{A}{B}}}{2},
	\end{equation}
	where 
	\begin{equation}\label{eq5577.332V08}
		\left (\alpha B\right )^2-4AB>0.
	\end{equation}
	\textbf{Case IV.} Suppose that
	\begin{align}\label{eq5577.332V10}
		e\pi&=\frac{C}{D}\\[.2cm]
		e^2\pi+ e^{-1}&=\alpha\in \mathbb{R}-\overline{\mathbb{Q}},\label{eq5577.332V10b}
	\end{align}
	where $\alpha>0$ is a transcendental number and $C, D>1$ are algebraic integers. Substitute \eqref{eq5577.332V10} into \eqref{eq5577.332V10b} and solve for $\pi$ in terms of $\alpha$ and $C,D>1$. This step yields\\
	\begin{equation}\label{eq5577.332V12}
		D^3 \pi^2-\alpha CD^2\pi+C^3=0
	\end{equation}
	and 
	\begin{equation}\label{eq5577.332V14}
		\pi=\frac{\alpha CD^2\pm\sqrt{(\alpha CD^2)^2-4C^3D^3}}{2B^3},
	\end{equation}
	where 
	\begin{equation}\label{eq5577.332V16}
		\left (\alpha CD^2\right )^2-4C^3D^3>0.
	\end{equation}
	Merging \eqref{eq5577.332V06} and \eqref{eq5577.332V14} yield
	\begin{equation}\label{eq5577.332V18}
		\frac{\alpha \pm\sqrt{\alpha^2-4\frac{A}{B}}}{2}=\frac{\alpha CD^2\pm\sqrt{(\alpha CD^2)^2-4C^3D^3}}{2B^3}.
	\end{equation}
	Clearly, the minimal polynomial $g(t)\in \mathbb{Z}[t]$ of $\alpha$ has algebraic integer coefficients. This implies that the transcendental number $\alpha>0$ satisfies an algebraic polynomial equation $g(\alpha)=0$. Ergo, the hypothesis stated in \eqref{eq5577.332V02} and \eqref{eq5577.332V10} is false.
	Therefore, the real numbers\\
	\begin{equation}
		e\pi, \qquad e+\pi, \qquad e^2\pi+ e^{-1}
	\end{equation}	
	are transcendentals. 
\end{proof}

\section{Results for Sums and Products of Powers of $e$ and $\pi$} \label{S2220-I}\hypertarget{S5566-I}
The real numbers $e$ and $\pi$ are transcendental numbers and it is immediate that at least one, the trace $Tr(\alpha)= e^2\pi^3+\pi^{-1}$ or the norm $N(\alpha)= (e\pi)^2$ of the polynomial 
\begin{equation}
	f_2(x)=(x- e^2\pi^3)(x-\pi^{-1})
\end{equation} 
is a transcendental number. Similarly, at least one, the trace $Tr(\alpha)= e^3\pi^2+e^{-1}$ or the norm $N(\alpha)= (e\pi)^2$ of the polynomial 
\begin{equation}
	f_3(x)=(x- e^3\pi^2)(x-e^{-1})
\end{equation}
is a transcendental number.\\

An elementary proof that all these traces and norms are transcendental numbers is given here. 

\begin{theorem}\label{thm5566.332V}\hypertarget{thm5566.332V} The real numbers
	\begin{equation}
		(e\pi)^2, \qquad e^2\pi^3+ \pi^{-1}, \qquad e^3\pi^2+ e^{-1}
	\end{equation}	
	are transcendental numbers. In particular, $e\pi$ is a transcendental number.	
\end{theorem}
\begin{proof} The two systems of equations associated with the traces and norms of the polynomials $f_2(x)$ and $f_3(x)$ are treated separately. \\
	
	\textbf{Case I.} Suppose that 
	\begin{align}\label{eq5566.332V01}
		(e\pi)^2&=\alpha\in \mathbb{R}-\overline{\mathbb{Q}}\\[.2cm]
		e^2\pi^3+ \pi^{-1}&=\frac{A}{B},
	\end{align}
	where $\alpha>0$ is a transcendental number and $A, B>1$ are algebraic integers.  Substitute \eqref{eq5566.332V01} into (1.5) and solve for $\pi$ in terms of $\alpha$ and $A,B>1$. This step yields\\
	\begin{equation}\label{eq5566.332V03}
		\alpha \pi^2-\frac{A}{B}\pi+1=0
	\end{equation}
	and 
	\begin{equation}\label{eq5566.332V05}
		\pi=\frac{\frac{A}{B}\pm\sqrt{\left (\frac{A}{B}\right )^2-4\alpha}}{2\alpha},
	\end{equation}
	where 
	\begin{equation}\label{eq5566.332V09}
		\left (\frac{A}{B}\right )^2-4\alpha>0.
	\end{equation}
	\textbf{Case II.} Suppose that
	\begin{align}\label{eq5566.332V11}
		(e\pi)^2&=\alpha\in \mathbb{R}-\overline{\mathbb{Q}}\\[.2cm]
		e^3\pi^2+ e^{-1}&=\frac{C}{D},
	\end{align}
	where $\alpha>0$ is a transcendental number and $C, D>1$ are algebraic integers. Substitute \eqref{eq5566.332V05} into (1.8) and solve for $\pi$ in terms of $\alpha$ and $C,D>1$. This step yields\\
	
	\begin{equation}\label{eq5566.332V13}
		\frac{1}{\alpha^{1/2}}\pi^2-\frac{C}{D}\pi+ \alpha^{3/2}=0
	\end{equation}
	and 
	\begin{equation}\label{eq5566.332V15}
		\pi=\frac{2\alpha^{3/2}}{\frac{C}{D}\pm\sqrt{\left (\frac{C}{D}\right )^2-4\alpha}},
	\end{equation}
	where 
	\begin{equation}\label{eq5566.332V21}
		\left (\frac{C}{D}\right )^2-4\alpha>0.
	\end{equation}
	Merging \eqref{eq5566.332V05} and \eqref{eq5566.332V15} yield
	\begin{equation}\label{eq5566.332V23}
		\frac{\frac{A}{B}\pm\sqrt{\left (\frac{A}{B}\right )^2-4\alpha}}{2\alpha}=\frac{2\alpha^{3/2}}{\frac{C}{D}\pm\sqrt{\left (\frac{C}{D}\right )^2-4\alpha}}.
	\end{equation}
	Clearly, the minimal polynomial $f(t)\in \mathbb{Z}[t]$ of $\alpha$ has algebraic integer coefficients. This implies that the transcendental number $\alpha>0$ satisfies an algebraic polynomial equation $f(\alpha)=0$. Ergo, the hypothesis stated in \eqref{eq5566.332V01} and \eqref{eq5566.332V11} is false.
	For example, setting $a=A/B$, $b=C/D$ and $t=\alpha$. This change of variables reduces \eqref{eq5566.332V23} to
	\begin{eqnarray}
		t^5&=&\left (a\pm\sqrt{a^2 - 4 t}\right )^2\left (b\pm\sqrt{b^2 - 4 t}\right )^2\\
		&=&-4 a^2 b \sqrt{b^2 - 4 t)}- 4 a b^2 \sqrt{a^2 - 4 t} + 4 a b \sqrt{a^2 - 4 t} \sqrt{b^2 - 4 t} \nonumber\\
		&&+ 4 a^2 b^2 - 8 a^2 t + 8 a t \sqrt{a^2 - 4 t} - 8 b^2 t + 8 b t \sqrt{b^2 - 4 t} + 16 t^2\nonumber
	\end{eqnarray}
	
	Collecting terms on the left and squaring yield
	\begin{eqnarray}
		\left (t^5-4 a^2 b^2 + 8 a^2 t- 16 t^2+ 8 a^2 t\right )^2
		&=&\left ((8bt-4 a^2 b) \sqrt{b^2 - 4 t)}\right .\nonumber\\
		&&\;+(8at- 4 a b^2) \sqrt{a^2 - 4 t} \nonumber\\
		&&\left .\;+ 4 a b \sqrt{a^2 - 4 t} \sqrt{b^2 - 4 t}\;\right )^2 .
	\end{eqnarray}
	Continuing this process of collecting terms on the left and squaring ends up with a polynomial $f(t)$ over the algebraic integers.\\
	
	Reversing the hypothesis also leads to two systems of equations associated with the traces and norms. These are treated separately. \\
	
	\textbf{Case III.} Suppose that 
	\begin{align}\label{eq5566.332V02}
		(e\pi)^2&=\frac{A}{B}\\[.2cm]
		e^2\pi^3+ \pi^{-1}&=\alpha\in \mathbb{R}-\overline{\mathbb{Q}},
	\end{align}
	where $\alpha>0$ is a transcendental number and $A, B>1$ are algebraic integers.  Substitute \eqref{eq5566.332V02} into (1.18) and solve for $\pi$ in terms of $\alpha$ and $A,B>1$. This step yields\\
	\begin{equation}\label{eq5566.332V04}
		\frac{A}{B} \pi^2-\alpha\pi+1=0
	\end{equation}
	and 
	\begin{equation}\label{eq5566.332V06}
		\pi=\frac{\alpha B\pm\sqrt{\left (\alpha B\right )^2-4AB}}{2A},
	\end{equation}
	where 
	\begin{equation}\label{eq5566.332V08}
		\left (\alpha B\right )^2-4AB>0.
	\end{equation}
	\textbf{Case IV.} Suppose that
	\begin{align}\label{eq5566.332V10}
		(e\pi)^2&=\frac{C}{D}\\[.2cm]
		e^3\pi^2+ e^{-1}&=\alpha\in \mathbb{R}-\overline{\mathbb{Q}},
	\end{align}
	where $\alpha>0$ is a transcendental number and $C, D>1$ are algebraic integers. Substitute \eqref{eq5566.332V10} into (1.23) and solve for $\pi$ in terms of $\alpha$ and $C,D>1$. This step yields\\
	\begin{equation}\label{eq5566.332V12}
		e^2 \pi^2-\frac{1}{\alpha}\pi-\frac{C}{D}=0
	\end{equation}
	and 
	\begin{equation}\label{eq5566.332V14}
		\pi=\frac{2C\sqrt{\frac{C}{D}}}{\alpha D\pm\sqrt{\alpha^2D^2-4CD}},
	\end{equation}
	where 
	\begin{equation}\label{eq5566.332V16}
		\left (\alpha D\right )^2-4CD>0.
	\end{equation}
	Merging \eqref{eq5566.332V06} and \eqref{eq5566.332V14} yield
	\begin{equation}\label{eq5566.332V18}
		\frac{\alpha B\pm\sqrt{\left (\alpha B\right )^2-4AB}}{2A}=\frac{2C\sqrt{\frac{C}{D}}}{\alpha D\pm\sqrt{\alpha^2D^2-4CD}}.
	\end{equation}
	Clearly, the minimal polynomial $g(t)\in \mathbb{Z}[t]$ of $\alpha$ has algebraic integer coefficients. This implies that the transcendental number $\alpha>0$ satisfies an algebraic polynomial equation $g(\alpha)=0$. Ergo, the hypothesis stated in \eqref{eq5566.332V02} and \eqref{eq5566.332V10} is false.
	Therefore, the real numbers\\
	\begin{equation}
		(e\pi)^2, \qquad e^2\pi^3+ \pi^{-1}, \qquad e^3\pi^2+ e^{-1}
	\end{equation}	
	are transcendentals. In particular, $e\pi$ is transcendental.
\end{proof}

\section{Irrationality of the Numbers $\alpha +\pi$} \label{S1020-I}\hypertarget{S1020-I}
The real numbers $e$ and $\pi$ are irrational numbers and it is immediate that at least one, the trace $Tr(\alpha)=e+\pi$ or the norm $N(\alpha)=e\pi$ of the polynomial $f(x)=(x+e)(x+\pi)$, is irrational. The above example is a special case proved in \hyperlink{cor1020.351V}{Lemma} \ref{cor1020.351V} of a more general sum of two irrational numbers considered here.

\begin{thm} \label{thm1020.350V}\hypertarget{thm1020.350V} If $\alpha>0$ is an irrational real number, then $$\alpha +\pi$$ 
	is irrational.
\end{thm}

\begin{proof}[\textbf{Proof}] To derive a reductio ad absurdum, assume the real number $\alpha +\pi=A/B\in \mathbb{Q}$ is rational, where $A\geq A_0,B\geq B_0\in \mathbb{Z}^{\times}$ are fixed integers. Let $\{p_n,q_n: n\geq1\}$ be the sequence of convergents of the continued fraction of $\alpha>0$, and rewrite it in the form
	
	\begin{equation}\label{eq1020.350Vf}
		B\pi =A-B\alpha\geq p_{2n+1}-\alpha q_{2n+1}>0,
	\end{equation}
where $n\geq n_0$ is sufficiently large. This follows from \hyperlink{lem2000.07}{Lemma} \ref{lem2000.07}. Computing the absolute value of the sine of both sides yields
	\begin{equation}\label{eq1020.350Vi}
	|	\sin(B\pi) |=|\sin (A-B\alpha)|\geq |\sin(p_{2n+1}-\alpha q_{2n+1})|
	\end{equation}
for sufficiently large $n\geq n_0$. This follows from the fact that $|\sin (A-B\alpha)|>0$ is fixed real number but $|\sin(p_{2n+1}-\alpha q_{2n+1})| \to 0$ as $n\to \infty$.  \\

Evaluating the left side and using $|z|\ll\sin(z)|\ll |z|$ for $0\leq |z|<1$ on the right side yield 
	\begin{eqnarray}\label{eq1020.350Vj}
		0&=&\bigg|\sin(B\pi)\bigg| \\[.3cm]
		&=&\bigg|\sin (A-B\alpha)\bigg|\nonumber \\[.3cm]
		&\geq &\bigg|\sin(p_{2n+1}-\alpha q_{2n+1})\bigg|\nonumber \\[.3cm]
			&\gg &\bigg|p_{2n+1}-\alpha q_{2n+1}\bigg|\nonumber \\[.3cm]
		&\geq& \frac{1}{q_{2n+2}}\nonumber,
	\end{eqnarray}
where the last line follows from \hyperlink{lem2000.05}{Lemma} \ref{lem2000.05}.	Clearly, for sufficiently large $n\geq n_0$ this is a contradiction. Ergo, the real number $\alpha +\pi\ne A/B$ is not rational.
\end{proof}

\begin{cor}\label{cor1020.351V}\hypertarget{cor1020.351V} The real number $e +\pi$ is irrational.
\end{cor}
\begin{proof}[\textbf{Proof}] Assume the real number $e +\pi=A/B\in \mathbb{Q}$ is rational, where $A>20,B\geq5$ are fixed integers. Let $\{p_n,q_n: n\geq1\}$ be the sequence of convergents of the continued fraction of $e$, and rewrite it in the form
	
	\begin{equation}\label{eq1020.351Vf}
		B\pi =A-Be\geq p_{2n+1}- q_{2n+1}e>0,
	\end{equation}
	where $n\geq n_0$ is sufficiently large. This follows from \hyperlink{lem2000.07}{Lemma} \ref{lem2000.07}. Computing the absolute value of the sine of both sides yields
	\begin{equation}\label{eq1020.351Vi}
		|	\sin(B\pi) |=|\sin (A-Be)|\geq |\sin(p_{2n+1}- q_{2n+1}e)|
	\end{equation}
for sufficiently large $n\geq n_0$.	This follows from the fact that $|\sin (A-Be)|>0$ is a fixed real number but $|\sin(p_{2n+1}- q_{2n+1}e)| \to 0$ as $n\to \infty$.  \\

Evaluating the left side and using $|z|\ll\sin(z)|\ll |z|$ for $0\leq |z|<1$ on the right side yield 
	\begin{eqnarray}\label{eq1020.351Vj}
		0&=&\bigg|\sin(B\pi)\bigg| \\[.3cm]
		&=&\bigg|\sin (A-Be)\bigg|\nonumber \\[.3cm]
		&\geq &\bigg|\sin(p_{2n+1}-q_{2n+1}e)\bigg|\nonumber \\[.3cm]
		&\gg &\bigg|p_{2n+1}- q_{2n+1}e\bigg|\nonumber \\[.3cm]
		&\geq& \frac{1}{q_{2n+2}}\nonumber,
	\end{eqnarray}
	where the last line follows from \hyperlink{lem2000.05}{Lemma} \ref{lem2000.05}.	Clearly, for sufficiently large $n\geq n_0$ this is a contradiction. Ergo, the real number $e +\pi\ne A/B$ is not rational.
\end{proof}
\begin{table}[H]
	\centering
\begin{tabular}{|l|l|l|l|l|}
	\hline
	$n$ & Convergent $p_n/q_n$ & Approx. $p_n/q_n$ & Error $e - p_n/q_n$ & $|q_n e - p_n|$ \\ \hline
	1 & 2/1 & 2.00000000 & +0.71828183 & 0.71828183 \\ \hline
	2 & 3/1 & 3.00000000 & -0.28171817 & 0.28171817 \\ \hline
	3 & 8/3 & 2.66666667 & +0.05161516 & 0.15484549 \\ \hline
	4 & 11/4 & 2.75000000 & -0.03171817 & 0.12687268 \\ \hline
	5 & 19/7 & 2.71428571 & +0.00399611 & 0.02797280 \\ \hline
	6 & 87/32 & 2.71875000 & -0.00046817 & 0.01498144 \\ \hline
	7 & 106/39 & 2.71794872 & +0.00033311 & 0.01299131 \\ \hline
	8 & 193/71 & 2.71830986 & -0.00002803 & 0.00199013 \\ \hline
	9 & 1264/465 & 2.71827957 & +0.00000226 & 0.00105065 \\ \hline
	10 & 1457/536 & 2.71828358 & -0.00000175 & 0.00093952 \\ \hline
	11 & 2721/1001 & 2.71828172 & +0.00000011 & 0.00011113 \\ \hline
	12 & 23225/8544 & 2.71828184 & -0.00000001 & 0.00006450 \\ \hline
	13 & 25946/9545 & 2.71828182 & +0.00000001 & 0.00005776 \\ \hline
	14 & 49171/18089 & 2.71828183 & -0.00000000 & 0.00000674 \\ \hline
	15 & 517656/190425 & 2.71828183 & +0.00000000 & 0.00000305 \\ \hline
	16 & 566827/208514 & 2.71828183 & -0.00000000 & 0.00000277 \\ \hline
	17 & 1084483/398939 & 2.71828183 & +0.00000000 & 0.00000028 \\ \hline
	18 & 13580623/4996082 & 2.71828183 & -0.00000000 & 0.00000011 \\ \hline
	19 & 14665106/5395021 & 2.71828183 & +0.00000000 & 0.00000010 \\ \hline
	20 & 28245729/10391103 & 2.71828183 & -0.00000000 & 0.00000001 \\ \hline
\end{tabular}
	\caption{First few convergents of $e$ and their errors.}
\end{table}
\vskip .1 in 
\begin{table}[H]
	\centering
\begin{tabular}{|l|l|l|l|l|}
	\hline
	$n$ & Convergent $p_n/q_n$ & Approx. $p_n/q_n$ & Error $(e+\pi) - p_n/q_n$ & $|q_n(e+\pi) - p_n|$ \\ \hline
	1 & 5/1 & 5.00000000 & +0.85987448 & 0.85987448 \\ \hline
	2 & 6/1 & 6.00000000 & -0.14012552 & 0.14012552 \\ \hline
	3 & 41/7 & 5.85714286 & +0.00273163 & 0.01912138 \\ \hline
	4 & 47/8 & 5.87500000 & -0.01512552 & 0.12100414 \\ \hline
	5 & 135/23 & 5.86956522 & -0.00969074 & 0.22288691 \\ \hline
	6 & 317/54 & 5.87037037 & -0.01049589 & 0.56677793 \\ \hline
	7 & 452/77 & 5.87012987 & -0.01025539 & 0.78966485 \\ \hline
	8 & 1673/285 & 5.87017544 & -0.01030096 & 2.93577239 \\ \hline
	9 & 2125/362 & 5.87016575 & -0.01029126 & 3.72543725 \\ \hline
	10 & 3798/647 & 5.87017002 & -0.01029553 & 6.66120964 \\ \hline
	11 & 110531/18862 & 5.85998303 & -0.00010855 & 2.04746237 \\ \hline
	12 & 114329/19509 & 5.86032088 & -0.00044640 & 8.70873491 \\ \hline
	13 & 224860/38371 & 5.86015480 & -0.00028032 & 10.75619728 \\ \hline
	14 & 1238629/211364 & 5.85987112 & +0.00000336 & 0.71018318 \\ \hline
	15 & 1463489/249735 & 5.85988015 & -0.00000567 & 1.41601409 \\ \hline
	16 & 2702118/461100 & 5.85993928 & -0.00006479 & 29.87633136 \\ \hline
	17 & 14974079/2555243 & 5.85987494 & -0.00000046 & 1.17540201 \\ \hline
	18 & 17676197/3016343 & 5.85987560 & -0.00000112 & 3.37817454 \\ \hline
	19 & 32650276/5571586 & 5.85987530 & -0.00000081 & 4.55357655 \\ \hline
	20 & 50326473/8587929 & 5.85987540 & -0.00000092 & 7.93175109 \\ \hline
\end{tabular}
	\caption{First few convergents of $e+\pi$ and their errors.}
\end{table}
Another approach is a simple applications of the above linear independence results demonstrate that both of these numbers are irrational, see Corollary \ref{cor9570.25}, and Corollary \ref{cor4035.23} respectively. Similar application demonstrates that the real number
\begin{equation}
	\pi+\pi^2
\end{equation}
is irrational, see Corollary \ref{cor4055.23}.

\section{The Irrational Limit Test}\label{s3288}\hypertarget{s3288}
The \textit{\textit{irrational limit test}} converts some apparently intractable decision problems in the real domain $\R$ to simpler decision problems in the finite field domain $\F_2=\{0,1\}$. 

\begin{dfn} 
{\normalfont Let $\alpha \in \R$ be a real number. The \textit{irrational limit test} is a map $\mathcal{I}: \R \longrightarrow \F_2=\{0,1\}$ defined by
\begin{equation}\label{eq3288.44}
\mathcal{I}(\alpha)=
 \lim_{x \to \infty}\frac{1}{2x} \sum_{-x \leq n \leq x}e^{i\alpha n}.
\end{equation}
}
\end{dfn}
The normalization is intrinsic to the number $\pi$. But, it can be modified as needed. The \textit{irrational limit test} is a point map or equivalently a class map, and it is not invertible. But, inversion is not required in the applications to decision problems. 

\begin{lem}\label{lem3288.06} For any real number $\alpha \in \R$, the \textit{irrational limit test} satisfies the followings.
\begin{equation}\label{eq3288.37}
\mathcal{I}(2 \pi m\alpha)=
\begin{cases}
1 & \text{ if and only if } \alpha \in \Q,\text{ for some } m \in \Z^{\times},\\
0 & \text{ if and only if } \alpha \notin \Q,\text{ for any } m \in \Z^{\times}.
\end{cases}
\end{equation}
\end{lem}

\begin{proof} Given any rational number $\alpha \in \Q$, there is an integer $m \in \Z$ such that $\alpha m \in \Z$, and the limit is
\begin{equation}\label{eq3288.47}
\mathcal{I}(2 \pi m\alpha)=\lim_{x \to \infty}\frac{1}{2x} \sum_{-x\leq n \leq x}e^{i2  \pi \alpha mn}
=\lim_{x \to \infty}\frac{1}{2x} \sum_{-x\leq n \leq x}1=1.
\end{equation}
The above proves that for any rational number $\alpha \in \Q$, and any integer $m$, the  sequence 
\begin{equation}
\{\alpha mn: n \in \Z\}
\end{equation} 
is not uniformly distributed. While for any irrational number $\alpha \notin \Q$, and any integer $m\ne 0$, the sine function $\sin(\alpha \pi m)\ne 0$. Hence, applying Lemma \ref{lem5489.32}, the evaluation of the limit is  
\begin{eqnarray}\label{eq3288.76}
\mathcal{I}( 2\pi \alpha m )&=& \lim_{x \to \infty}\frac{1}{2x} \sum_{-x \leq n \leq x}e^{i2\pi \alpha m n} \nonumber \\
&\leq &\lim_{x \to \infty}\frac{1}{2x} \frac{1}{\left | \sin\left (\alpha \pi  m \right )\right |} \\
&=&0 \nonumber.
\end{eqnarray}
The above proves that for any irrational number $\alpha \in \Q$, and any integer $m\ne 0$, the sequence 
\begin{equation}
\{\alpha mn: n \in \Z\},
\end{equation}
 is uniformly distributed. This proof is equivalent to the Weyl criterion, see \cite[Theorem 2.1]{KN74}.
\end{proof} 
As it is evident, the class function $\mathcal{I}$ maps the class of rational numbers $\Q$ to $1$ and the class of irrational numbers $\I=\R-\Q$ to $0$. The \textit{irrational limit test} induces an equivalence relation on the set of real numbers $\R$:
\begin{itemize}
 \item A pair of real numbers $a$ and $b$ are equivalent $a \sim b$ if and only if $\mathcal{I}(2 \pi a)=\mathcal{I}(2 \pi b)$. This occurs if either both $a$ and $b$ are rational numbers or both $a$ and $b$ are irrational numbers.
 \item A pair of real numbers $a$ and $b$ are not equivalent $a \not \sim b$ if and only if $\mathcal{I}(2 \pi a)\ne\mathcal{I}(2 \pi b)$. This occurs if either $a$ or $b$ is a rational numbers but not both.
\end{itemize}
Some standard irrationality tests, criteria, and proofs are given in \cite[Chapter 7]{AZ98}, \cite{BA75}, \cite{HW08}, \cite{NI47}, \cite{SJ05}, \cite{WM00}, et alii.

\section{Approximation By Lattice Points}\label{s4422}\hypertarget{s4422}
A handful of elementary integer relations and integers points approximations are considered in this Section.
\begin{lem}\label{lem4422.10} The numbers $e$ and $\pi$ are not integer multiple. Specifically, $ek \ne  m\pi $ for any integers $k,m \geq1$. 
\end{lem}
\begin{proof} Numerically $\sin(e)=0.410781\ldots \ne 0$. Computing it using the infinite product yields  
\begin{equation}\label{eq4422.12}
0\ne \sin \left (e\right ) =\frac{1}{e} \prod_{n \geq1} \left (1- \frac{e^2}{\pi^{2}n^2} \right ) .
\end{equation} 
Ergo, for each integer $n\geq1$, the local factor
\begin{equation}\label{eq4422.14}
1- \frac{e^2}{\pi^{2}n^2}=1- \left(\frac{e}{\pi n}\right)^2\ne 0
\end{equation} 
cannot vanish. This proves that $e/\pi n\ne 1$. Equivalently, these numbers are not integer multiple: $ek\ne m\pi $ for any integers $k,m\geq1$. 
\end{proof}

The discrete lines 
\begin{equation}\label{eq4422.740}
\mathcal{L}_0(r)=\{(2r+1)\pi/2:r \in\Z^{\times}\} \quad\quad \text{ and }  \quad\quad\mathcal{L}_1(r)=\{(2r+1)\pi:r \in\Z^{\times}\}
\end{equation} 
never intercept the discrete lattice
\begin{equation}\label{eq4422.743}
\mathcal{L}(k,m)=\{ke+m: k,m\in  \Z^{\times}\},
\end{equation}
but comes arbitrarily close. A proof, based on the simplest form of the Kronecker approximation theorem, see Theorem \ref{thm2000.990}, is given below.

\begin{lem}\label{lem4422.700} If $k$ and $m $ are nonzero integers, and let $r \in \Z$, then 
\begin{multicols}{2}
 \begin{enumerate} \label{eq4422.770}
\item$ \displaystyle 
 ke+m\ne (2r+1)\pi/2. $
\item$\displaystyle
ke+m\ne r\pi.
$
\end{enumerate}
\end{multicols}
\end{lem}
\begin{proof} (i) It is sufficient to verify the inequality (\ref{eq4422.770})-i on the first quadrant, which is specified by $k\geq 1$ and $m\geq1$. The verification in any quadrant is almost the same. Let $\{p_n/q_n:n \geq 1\}$ be the sequence of convergents of the irrational number $e$. The Diophantine approximation inequalities 
\begin{equation}\label{eq4422.775}
\frac{1}{2q_{n+1}}\leq \left | q_ne-p_n \right |\leq \frac{1}{q_{n}}
\end{equation} 
and
\begin{equation}\label{eq4422.776}
\left | q_ne-p_n \right |\leq \left |ke-m\right |
\end{equation} 
for $k\le q_n$, 
see Lemma \ref{lem2000.05}, and Lemma \ref{lem2000.07}, lead to the lattice points approximation
\begin{eqnarray}\label{eq4422.780}
\left | ke+m-(2r+1)\pi/2 \right |
&\geq&\left | \left |ke-m\right | -(2r+1)\pi/2 \right | \\
&\geq&\left | \left |q_ne-p_n\right |-(2r+1)\pi/2 \right |\nonumber,
\end{eqnarray} 
where $r\in \Z$, and $|2r+1|\geq 1$. Rearranging it, and applying the reverse triangle inequality $|X-Y|\geq ||X|-|Y||$,  yield
\begin{eqnarray}\label{eq4422.782}
\left | \left | e-\frac{p_n}{q_n}\right |-\left | \frac{(2r+1)\pi}{2q_n} \right |\right |
&\geq&\left | \left |\frac{1}{2q_{n+1}}\right |-\left | \frac{(2r+1)\pi}{2q_n} \right |\right |\\
&\geq&\frac{1}{2q_{n+1}}\nonumber\\
&>&0\nonumber.
\end{eqnarray}
Therefore, relation $ke+m= (2r+1)\pi/2$ is false for any nontrivial integer point $(k,m)=(k\ne0,m\ne0)$ and $r \in \Z$. (ii) The proof for this case is similar.
\end{proof}

Observe that the continued fraction $e=[a_0,a_1,a_2,\ldots]$ has arbitrary long arithmetic progressions, very visibly, see Theorem \ref{thm2000.990}, the relation $ke+m= (2r+1)\pi/2$ would implies that continued fraction $\pi=[b_0,b_1,b_2,\ldots]$ has arbitrary long arithmetic progressions. But, this is unknown. These elementary results seem to be implied by a more advanced technique given in \cite{RG73} about certain equivalence of irrational numbers.

\begin{lem}\label{lem4422.800} Let $1\leq u<v$ be a pair of integers, and let $k$, $m$, and $r $ be any nonzero integers, then 
\begin{multicols}{2}
 \begin{enumerate} \label{eq4422.870}
\item$ \displaystyle 
 k\pi^u+m\pi^v\ne r\pi. $
\item$\displaystyle
ke^u+me^v\ne r\pi.
$
\end{enumerate}
\end{multicols}
\end{lem}
\begin{proof} (i) It is sufficient to verify the inequality (\ref{eq4422.870})-i on the first quadrant, which is specified by $k\geq 1$ and $m\geq1$. The verification in any quadrant is almost the same. Let $\{p_n/q_n:n \geq 1\}$ be the sequence of convergents of the irrational number $\pi^{v-u}$. The Diophantine approximation inequalities 
\begin{equation}\label{eq4422.875}
\frac{1}{2q_{n+1}}\leq \left | q_n\pi^{v-u}-p_n \right |\leq \frac{1}{q_{n}}
\end{equation} 
and
\begin{equation}\label{eq4422.876}
\left | q_n\pi^{v-u}-p_n \right |\leq \left |k\pi^{v-u}-m\right |
\end{equation} 
for $k\leq q_n$, see Lemma \ref{lem2000.05}, and Lemma \ref{lem2000.07}, lead to the lattice points approximation
\begin{eqnarray}\label{eq4422.880}
\left | k\pi^{u}+m\pi^{v}-r\pi \right |
&=&\left | \pi^{u}\left (k\pi^{v-u}+m\right )-r\pi \right |\\
&\geq&\left | \pi^{u}\left |k\pi^{v-u}-m\right | -r\pi \right | \nonumber\\
&\geq&\left | \pi^{u}\left |q_n\pi^{v-u}-p_n\right |-r\pi \right |\nonumber,
\end{eqnarray} 
where $r\in \Z$, and $\pi^{u}\geq 1$. Rearranging it, and applying the reverse triangle inequality $|X-Y|\geq ||X|-|Y||$,  yield
\begin{eqnarray}\label{eq4422.882}
\left | \pi^{u}\left | \pi^{v-u}-\frac{p_n}{q_n}\right |-\left | \frac{r\pi}{q_n} \right |\right |
&\geq&\left | \left |\frac{\pi^{u}}{2q_{n+1}}\right |-\left | \frac{r\pi}{q_n} \right |\right |\\
&\geq&\frac{1}{2q_{n+1}}\nonumber\\
&>&0\nonumber.
\end{eqnarray}
Therefore, relation $k\pi^{u}+m\pi^{v}= r\pi$ is false for any nontrivial integer point $(k,m)=(k\ne0,m\ne0)$ and $r \in \Z$. (ii) The proof for this case is similar.
\end{proof}

\section{Nonvanishing Sine Function Values}\label{s5589}\hypertarget{s5589}
The nonvanishing of the sine function at certain real numbers are required in the proofs of certain results. These are verified using either the irrationality of the real number $\pi$ or via the infinite product $\sin \left (x\right ) =x^{-1}\prod_{n \geq1} \left (1- (x\pi^{-1}n^{-1})^2 \right ) $ for any real number $x\in \R$ or the reflection formula $\Gamma \left (1-z \right )\Gamma \left (z \right )=\pi/\sin \pi z$ of the gamma function $\Gamma(z)$ for $z\in \C$.

\begin{lem}\label{lem5589.10} For any pair of integers $r\ne 1$, and $m \ne0$, the sine function satisfies the followings.
\begin{multicols}{2}
 \begin{enumerate} \label{eq5589.11}
\item $ \displaystyle \sin\left ( m\right )\ne0. $
\item $\displaystyle \sin\left ( \pi^rm\right )\ne0.$
\end{enumerate}
\end{multicols} 
\end{lem}
\begin{proof} (i) The verification, using the reflection formula of the gamma function, yields
\begin{eqnarray}\label{eq5589.11c}
\sin (m)&=&\sin \left (\pi \cdot m/\pi\right )\\
&=&\frac{\pi}{\Gamma \left (1-m/\pi \right )\Gamma \left (m/\pi \right )}\nonumber\\
&\ne&0 \nonumber
\end{eqnarray}
for any integer $m\geq1$ since the gamma function $\Gamma(z)$ has its poles at the negative integers $z=n\leq0$, and
\begin{equation} \label{eq5589.13}
1-\frac{m}{\pi} \qquad \text{ and } \qquad \frac{m}{\pi}
\end{equation}
are irrational numbers, not negative integers. (ii) Similar to the previous case. 
\end{proof}

\begin{lem}\label{lem5589.70} If $k$ and $m $ are nonzero integers, then 
\begin{equation}\label{eq5589.70}
\sin(ke+m)\ne 0.
\end{equation} 
\end{lem}
\begin{proof} The task to prove that the set of nontrivial integer solutions $(k,m)\ne(0,0)$ of the equation
\begin{equation}\label{eq5589.72}
\sin(ke+m)= 0
\end{equation} 
is empty splits into three different cases.\\

\textbf{Case 1.} $k= 0$, and $m\ne0$. The relation 
\begin{eqnarray}\label{eq5589.78}
 \sin \left ( ke+m\right ) &=& \sin \left (m\right ) \\
&\ne&0\nonumber
\end{eqnarray} 
is true, see Lemma \ref{lem5589.10}.\\

\textbf{Case 2.} $k\ne 0$, and $m=0$. By Lemma \ref{lem4422.10}, $e\ne a\pi$ for any integer $a \geq 1$. Thus, the multiple $ke\ne ak \pi=n\pi$. This implies that the equation  
\begin{equation}\label{eq5589.74}
\sin(ke+m)=\sin(ke)=\sin(n\pi)=0
\end{equation} 
is impossible.

\textbf{Case 3.} $k\ne 0$, and $m\ne0$. By Lemma \ref{lem4422.700}, $ke+m\ne r\pi$ for any integers $k$, $m$, and $r \in \Z$.  This implies that the equation
\begin{equation}\label{eq5589.74}
\sin(ke+m)=\sin\left (r\pi\right )
\end{equation} 
is impossible. 
\end{proof}

\begin{lem}\label{lem5589.170} If $k$ and $m $ are nonzero integers, then 
\begin{equation}\label{eq5589.170}
\sin(ke\pi+m)\ne 0.
\end{equation} 
\end{lem}
\begin{proof} The task to prove that the set of nontrivial integer solutions $(k,m)\ne(0,0)$ of the equation
\begin{equation}\label{eq5589.172}
\sin(ke\pi+m)= 0
\end{equation} 
is empty splits into three different cases.\\

\textbf{Case 1.} $k= 0$, and $m\ne0$. The relation 
\begin{eqnarray}\label{eq5589.178}
 \sin \left ( ke\pi+m\right ) &=& \sin \left (m\right ) \\
&\ne&0\nonumber
\end{eqnarray} 
is true, see Lemma \ref{lem5589.10}.\\
     
\textbf{Case 2.} $k\ne 0$, and $m=0$. Since $e$ is irrational, the relation $ke\pi = n\pi$, where $n\ne 0$, is impossible. This implies that the equation  
\begin{equation}\label{eq5589.174}
\sin(ke\pi+m)=\sin(ke\pi)=\sin(n\pi)=0
\end{equation} 
is impossible.\\

\textbf{Case 3.} $k\ne 0$, and $m\ne0$. By Lemma \ref{lem4422.700}, $ke\pi+m\ne (2r+1)\pi$ for any integers $k$, $m$, and $r \in \Z$.  This implies that the equation
\begin{equation}\label{eq5589.74i}
\sin(ke\pi+m)=\sin\left ((2r+1)\pi\right )
\end{equation} 
is impossible. 
\end{proof}

\begin{lem}\label{lem5589.370} Let $1\leq r<s$ be a pair of integers, and let $k$ and $m $ be nonzero integers, then 
\begin{equation}\label{eq5589.370}
\sin\left (k\pi^{r+1}+m\pi^{s+1}\right )\ne 0.
\end{equation} 
\end{lem}
\begin{proof} The task to prove that the set of nontrivial integer solutions $(k,m)\ne(0,0)$ of the equation
\begin{equation}\label{eq5589.372}
\sin\left (k\pi^{r+1}+m\pi^{s+1}\right )= 0
\end{equation} 
is empty splits into three different cases.\\

\textbf{Case 1.} $k= 0$, and $m\ne0$. Since $s\geq 2$ is an integer, the relation 
\begin{eqnarray}\label{eq5589.378}
 \sin \left (k\pi^{r+1}+m\pi^{s+1}\right ) &=& \sin \left (m\pi^{s+1}\right ) \\
&=&0\nonumber,
\end{eqnarray} 
where $\pi^{s+1}\geq \pi^3$, is false, see Lemma \ref{lem5589.10}.\\

\textbf{Case 2.} $k\ne 0$, and $m=0$. Since $r\geq 1$ is an integer, the relation 
\begin{equation}\label{eq5589.374j}
\sin\left (k\pi^{r+1}+m\pi^{s+1}\right )=\sin(k\pi^{r+1})=0
\end{equation} 
where $\pi^r\geq \pi^2$, is false, see Lemma \ref{lem5589.10}.\\     

\textbf{Case 3.} $k\ne 0$, and $m\ne0$. By Lemma \ref{lem4422.700}, $k\pi^{r+1}+m\pi^{s+1}\ne a\pi$ for any integers $k\ne0$, $m\ne0$, and $a \in \Z$. This implies that the equation
\begin{equation}\label{eq5589.374l}
\sin\left (k\pi^{r+1}+m\pi^{s+1}\right )=\sin\left (a\pi\right )
\end{equation} 
is impossible. 
\end{proof}

\section{Finite Sine Sums}\label{s5489}
\begin{lem}\label{lem5489.32} For any real number $t \ne k \pi$, $k \in \Z$, and a large integer $x \geq 1$, the finite sum
\begin{enumerate} [font=\normalfont, label=(\roman*)]
 \item $\displaystyle\sum_{-x\leq n \leq x}e^{i2 tn}=\frac{\sin((2x+1)t)}{\sin(t)}. $
\item  $\displaystyle\left |\sum_{-x\leq n \leq x}e^{i2 tn} \right | \leq\frac{1}{| \sin(t) |}. $
\end{enumerate}
\end{lem}
\begin{proof} (i) Expand the complex exponential sum into two subsums:
\begin{equation}\label{eq5489.50}
\sum_{-x\leq n \leq x}e^{i2 tn}=e^{-i2  t}\sum_{0\leq n \leq x-1}e^{-i2  tn}+\sum_{0\leq n \leq x}e^{i2 tn}.
\end{equation}
Lastly, use the geometric series to determine the closed form.	
\end{proof}


\section{Linear Independence Of $1$, $e$, and $\pi$} \label{s7575}

 \begin{proof} (Theorem \ref{thm7575.26}:) On the contrary, the numbers $1$, $e$ and $\pi$ are linearly dependent over the rational numbers, and the equation 
 \begin{equation} \label{eq7575.55}
1\cdot A+e \cdot B+ \pi\cdot C=0 ,
 \end{equation}
where $A, B, C \in \Z^{\times}$ are integers, has a nontrivial rational solution $(A,B,C)\ne(0,0,0)$. Rewrite it in the equivalent form 
\begin{equation} \label{eq5772.56}
-2 \pi C=2\left ( eB+ A \right ) .
\end{equation} 
 Take the \textit{irrational limit test}, see Lemma \ref{lem3288.06}, in both sides to obtain
\begin{equation} \label{eq7575.12}
\mathcal{I}(-2 \pi C)= \mathcal{I}\left (2(eB +A) \right ).
\end{equation}
The left side and the right side are evaluated separately.\\

\textbf{Left Side.} The verification is based on the identity $e^{-i2 \pi C}=1$, where $C$ is an integer. The evaluation of the limit is
\begin{equation}\label{eq7575.14}
\mathcal{I}(-2 \pi C)=\lim_{x \to \infty}\frac{1}{2x} \sum_{-x\leq n \leq x}e^{-i2  \pi Cn}=\lim_{x \to \infty}\frac{1}{2x} \sum_{-x\leq n \leq x}1=1.
\end{equation}

\textbf{Right Side.} The verification is based on the nonvanishing of the sine function $\sin\left  (eB +A  \right) \ne0$, see Lemma \ref{lem5589.70}. An application of Lemma \ref{lem5489.32} yields  

\begin{eqnarray}\label{eq7575.16}
\mathcal{I}\left (2(eB +A ) \right )&= &\lim_{x \to \infty}\frac{1}{2x} \sum_{-x \leq n \leq x}e^{i\left (2(eB +A) \right ) n}\\
&\leq &\lim_{x \to \infty}\frac{1}{2x}\frac{1}{\left |\sin\left (eB +A \right ) \right |} \nonumber\\
&=&0 \nonumber.
\end{eqnarray}
Clearly, these distinct evaluations
\begin{equation}\label{eq7575.49}
1=\mathcal{I}(-2 \pi C)\ne \mathcal{I}\left (2(eB +A ) \right )=0
\end{equation}
contradict equation (\ref{eq7575.12}). This implies that equation (\ref{eq7575.55}) does not have a nontrivial rational solution $\left (A,B,C\right) \in \Z^{\times}\times \Z^{\times}\times\Z^{\times}$. Hence, the numbers $1$, $ e$ and $\pi$ are linearly independent over the rational numbers $\Q^{\times}$. 
\end{proof}

\section{The Real Number $e+\pi$} \label{s9570}\hypertarget{s9570}
The continued fraction of the number understudy is
\begin{equation}\label{eq9570.50}
e+\pi=[5;1,6,7,3,21,2,1,2,2, 1,1,2,3,3,2,5,2,1,1,1,3,1, 8,\ldots].
\end{equation}  
The previous result immediately implies that this continued fraction is infinite.

\begin{cor} \label{cor9570.25}  The real number $e+\pi=5.859874\ldots $ is irrational number.
\end{cor}
 \begin{proof} By Theorem \ref{thm7575.26}, the equation 
 \begin{equation} \label{eq9570.55}
1\cdot A+e \cdot B+ \pi\cdot C=0 ,
 \end{equation}
has no nontrivial integer solutions $(A, B, C)\ne(0,0,0)$.
\end{proof}
\begin{conj} \label{cnj4035.04}  The real number $e+\pi$ is transcendental.
\end{conj}
\begin{conj} \label{conj9570.02}  The irrationality measure of the real number $e+\pi$ is $\mu(e+\pi)=2$.
\end{conj}
A few values were computed to illustrate the prediction in this conjecture, see Table \ref{t9570.04}. The fourth column displays the numerical approximation $\mu_0(e+\pi)$ of the actual value $\mu(e+\pi)$.

\begin{table}[h!]
\centering
\caption{Numerical Data For Irrationality Measure $|p_n/q_n-e-\pi|\geq q_n^{\mu(e+\pi)}$.} \label{t9570.04}
\begin{tabular}{l|l|l| l}
$n$&$p_n$&$q_n$&$\mu_0(e+\pi)$\\
\hline
1&$5$&   $1$   &$$\\
2&$6$&  $1$   &$$\\
3&$41$&   $7$   &$3.033470$\\
4&$93$&  $50?$   &$3.153443$\\
5&$920$&   $157$   &$2.608509$\\
6&$19613$&  $3347$   &$2.124717$\\
7&$40146$&  $6851$   &$2.382347$\\
8&$59759$&   $10198$   &$2.073126$\\
9&$379087$&  $64692$   &$2.067776$\\
10&$538751$&  $91939$   &$2.066541$\\ 
\end{tabular}
\end{table}

\section{Linear Independence Of $1$, $e$, And $\pi^{-1}$} \label{s8575}

\begin{proof} (Theorem \ref{thm8575.26}:) On the contrary, the numbers $1$, $e$ and $\pi^{-1}$ are linearly dependent over the rational numbers, and the equation 
\begin{equation} \label{eq8575.55}
1\cdot A+e \cdot B+ \pi^{-1}\cdot C=0 ,
\end{equation}
where $A, B, C \in \Z^{\times}$ are integers, has a nontrivial rational solution $(A,B,C)\ne(0,0,0)$. Rewrite it in the equivalent form 
\begin{equation} \label{eq8575.56}
-2 \pi A=2\left ( e\pi B+ C \right ) .
\end{equation} 
 Take the \textit{irrational limit test}, see Lemma \ref{lem3288.06}, in both sides to obtain
\begin{equation} \label{eq8575.12}
\mathcal{I}(-2 \pi A)= \mathcal{I}\left (2(e\pi B +C) \right ).
\end{equation}
The left side and the right side are evaluated separately.\\

\textbf{Left Side.} The verification is based on the identity $e^{-i2 \pi A}=1$, where $A$ is an integer. The evaluation of the limit is
\begin{equation}\label{eq8575.14}
\mathcal{I}(-2 \pi A)=\lim_{x \to \infty}\frac{1}{2x} \sum_{-x\leq n \leq x}e^{-i2  \pi An}=\lim_{x \to \infty}\frac{1}{2x} \sum_{-x\leq n \leq x}1=1.
\end{equation}
\textbf{Right Side.} The verification is based on the nonvanishing $\sin\left  (e\pi B +C  \right) \ne0$ of the sine function, see Lemma \ref{lem5589.170}. An application of Lemma \ref{lem5489.32} yields  
\begin{eqnarray}\label{eq8575.16}
\mathcal{I}\left (2(e\pi B +C ) \right )&= &\lim_{x \to \infty}\frac{1}{2x} \sum_{-x \leq n \leq x}e^{i\left (2(e\pi B +C) \right ) n}\\
&\leq &\lim_{x \to \infty}\frac{1}{2x}\frac{1}{\left |\sin\left (e\pi B +C \right ) \right |} \nonumber\\
&=&0 \nonumber.
\end{eqnarray}
Clearly, these distinct evaluations
\begin{equation}\label{eq8575.49}
1=\mathcal{I}(-2 \pi A)\ne \mathcal{I}\left (2(e\pi B +C ) \right )=0
\end{equation}
contradict equation \eqref{eq8575.12}. This implies that equation \eqref{eq8575.55} does not have a nontrivial rational solution $\left (A,B,C\right) \in \Z^{\times}\times \Z^{\times}\times\Z^{\times}$. Hence, the numbers $1$, $ e$ and $\pi^{-1}$ are linearly independent over the rational numbers $\Q^{\times}$. 
\end{proof}

\section{The Real Number $e\pi$} \label{s4035}\hypertarget{s4935}
The continued fraction of the number understudy is
\begin{equation}\label{eq4035.50}
e\pi=[8;1,1,5,1,3,1,4,12,3,2,1,5,2,12,1,1,1,10,2,2,2,3,8,3, 2,2,2,29,1,\ldots].
\end{equation}  
The previous result immediately implies that this continued fraction is infinite.

\begin{cor} \label{cor4035.23}  The real number $e\pi=8.539734\ldots $ is irrational.
\end{cor}
 \begin{proof} By Theorem \ref{thm8575.26}, the equation 
 \begin{equation} \label{eq4035.55}
1\cdot A+e \cdot B+ \pi^{-1}\cdot C=0 ,
\end{equation}
has no nontrivial integer solutions $(A, B, C)\ne(0,0,0)$.
\end{proof}
\begin{conj} \label{cnj4035.04}  The real number $e\pi$ is transcendental.
\end{conj}
\begin{conj} \label{cnj4035.02}  The irrationality measure of the real number $e\pi$ is $\mu(e\pi)=2$.
\end{conj}
A few values were computed to illustrate the prediction in this conjecture, see Table \ref{t4035.06}. The fourth column displays the numerical approximation $\mu_0(e\pi)$ of the actual value $\mu(e\pi)$.

\begin{table}[h!]
\centering
\caption{Numerical Data For Irrationality Measure $|p_n/q_n-e\pi|\geq q_n^{\mu(e\pi)}$.} \label{t4035.06}
\begin{tabular}{l|l|l| l}
$n$&$p_n$&$q_n$&$\mu_0(e\pi)$\\
\hline
1&$8$&   $1$   &$$\\
2&$9$&  $1$   &$$\\
3&$17$&   $2$   &$4.653474$\\
4&$94$&  $11$   &$3.153443$\\
5&$111$&   $13$   &$2.599126$\\
6&$427$&  $50$   &$2.104500$\\
7&$538$&  $63$   &$2.382347$\\
8&$2579$&   $302$   &$2.442400$\\
9&$31486$&  $3687$   &$2.150201$\\
10&$97037$&  $11363$   &$2.123550$\\ 
\end{tabular}
\end{table}

\section{Linear Independence Of $1$, $\pi^r$, and $\pi^{s}$} \label{s8008}

\begin{proof} (Theorem \ref{thm8008.26}:) On the contrary, the numbers $1$, $\pi^r$ and $\pi^{s}$ are linearly dependent over the rational numbers, and the equation 
\begin{equation} \label{eq8008.55}
1\cdot A+\pi^r \cdot B+ \pi^{s}\cdot C=0 ,
\end{equation}
where $A, B, C \in \Z^{\times}$ are integers, has a nontrivial rational solution $(A,B,C)\ne(0,0,0)$. Rewrite it in the equivalent form 
\begin{equation} \label{eq8008.56}
-2 \pi A=2\left ( \pi^{r+1} B+ \pi^{s+1}C \right ) .
\end{equation} 
 Take the \textit{irrational limit test}, see Lemma \ref{lem3288.06}, in both sides to obtain
\begin{equation} \label{eq8008.12}
\mathcal{I}(-2 \pi A)= \mathcal{I}\left (2(\pi^{r+1} B+ \pi^{s+1}C) \right ).
\end{equation}
The left side and the right side are evaluated separately.\\

\textbf{Left Side.} The verification is based on the identity $e^{-i2 \pi A}=1$, where $A$ is an integer. The evaluation of the limit is
\begin{equation}\label{eq8008.14}
\mathcal{I}(-2 \pi A)=\lim_{x \to \infty}\frac{1}{2x} \sum_{-x\leq n \leq x}e^{-i2  \pi An}=\lim_{x \to \infty}\frac{1}{2x} \sum_{-x\leq n \leq x}1=1.
\end{equation}
\textbf{Right Side.} The verification is based on the nonvanishing $\sin\left  (\pi^{r+1} B+ \pi^{s+1}C  \right) \ne0$ of the sine function, see Lemma \ref{lem5589.370}. An application of Lemma \ref{lem5489.32} yields  
\begin{eqnarray}\label{eq8008.16}
\mathcal{I}\left (2(\pi^{r+1}\pi B+ \pi^{s+1}C ) \right )&= &\lim_{x \to \infty}\frac{1}{2x} \sum_{-x \leq n \leq x}e^{i\left (2(\pi^{r+1} B+ \pi^{s+1}C) \right ) n}\\
&\leq &\lim_{x \to \infty}\frac{1}{2x}\frac{1}{\left |\sin\left (\pi^{r+1} B+ \pi^{s+1}C \right ) \right |} \nonumber\\
&=&0 \nonumber.
\end{eqnarray}
Clearly, these distinct evaluations
\begin{equation}\label{eq8008.49}
1=\mathcal{I}(-2 \pi A)\ne \mathcal{I}\left (2(\pi^{r+1} B+ \pi^{s+1}C) \right )=0
\end{equation}
contradict equation (\ref{eq8008.12}). This implies that equation (\ref{eq8008.55}) does not have a nontrivial rational solution $\left (A,B,C\right) \in \Z^{\times}\times \Z^{\times}\times\Z^{\times}$. Hence, the numbers $1$, $\pi^r$ and $\pi^{s}$ are linearly independent over the rational numbers $\Q^{\times}$. 
\end{proof}

\section{The Real Number $\pi+\pi^2$} \label{s4055}\hypertarget{s4055}
The continued fraction of the number understudy is
\begin{equation}\label{eq4055.50}
\pi+\pi^2=[13,89,3,4,3,1,2,3,1,9,2,1,1,3,1,12,1,1,4,2748,6,91,18,19,2,12,1,\ldots].
\end{equation}  
The previous result immediately implies that this continued fraction is infinite.

\begin{cor} \label{cor4055.23}  The real number $\pi+\pi^2=13.011197\ldots $ is irrational.
\end{cor}
 \begin{proof} Set $r=1$ and $s=2$. By Theorem \ref{thm8008.26}, the equation 
 \begin{equation} \label{eq4055.55}
1\cdot A+\pi \cdot B+ \pi^2\cdot C=0 ,
\end{equation}
has no nontrivial integer solutions $(A, B, C)\ne(0,0,0)$.
\end{proof}

\begin{conj} \label{cnj4055.04}  The real number $\pi+\pi^2$ is transcendental.
\end{conj}
\begin{conj} \label{cnj4055.02}  The irrationality measure of the real number $\pi+\pi^2$ is $\mu(\pi+\pi^2)=2$.
\end{conj}
A few values were computed to illustrate the prediction in this conjecture, see Table \ref{t4055.06}. The fourth column displays the numerical approximation $\mu_0(\pi+\pi^2)$ of the actual value $\mu(\pi+\pi^2)$.

\begin{table}[h!]
\centering
\caption{Numerical Data For Irrationality Measure $|p_n/q_n-e\pi|\geq q_n^{\mu(\pi+\pi^2)}$.} \label{t4055.06}
\begin{tabular}{l|l|l| l}
$n$&$p_n$&$q_n$&$\mu_0(\pi+\pi^2)$\\
\hline
1&$13$&   $1$   &$$\\
2&$1158$&  $89$   &$2.262270$\\
3&$3487$&   $268$   &$2.273061$\\
4&$15106$&  $1161$   &$2.193714$\\
5&$48805$&   $3751$   &$2.068191$\\
6&$63911$&  $4912$   &$2.130031$\\
7&$176627$&  $13575$   &$2.152449$\\
8&$593792$&   $45637$   &$2.031677$\\
9&$770419$&  $59212$   &$2.211219$\\
10&$7527563$&  $578545$   &$2.073701$\\ 
\end{tabular}
\end{table}

\section{Basic Diophantine Approximations Results} \label{s2000}
All the materials covered in this section are standard results in the literature, see \cite{HW08}, \cite{LS95}, \cite{NZ91}, \cite{RH94}, \cite{SJ05}, \cite{WM00}, et alii.

\subsection{Rationals And Irrationals Numbers Criteria} 
A real number \(\alpha \in \mathbb{R}\) is called \textit{rational} if \(\alpha = a/b\), where \(a, b \in \mathbb{Z}\) are integers. Otherwise, the number
is \textit{irrational}. The irrational numbers are further classified as \textit{algebraic} if \(\alpha\) is the root of an irreducible polynomial \(f(x) \in
\mathbb{Z}[x]\) of degree \(\deg (f)>1\), otherwise it is \textit{transcendental}.\\

\begin{lem} \label{lem2000.01} If a real number \(\alpha \in \mathbb{R}\) is a rational number, then there exists a constant \(c = c(\alpha )\) such that
\begin{equation}
\frac{c}{q}\leq \left|  \alpha -\frac{p}{q} \right|
\end{equation}
holds for any rational fraction \(p/q \neq \alpha\). Specifically, \(c \geq  1/b\text{ if }\alpha = a/b\).
\end{lem}

This is a statement about the lack of effective or good approximations for any arbitrary rational number \(\alpha \in \mathbb{Q}\) by other rational numbers. On the other hand, irrational numbers \(\alpha \in \mathbb{R}-\mathbb{Q}\) have effective approximations by rational numbers. If the complementary inequality \(\left|  \alpha -p/q \right| <c/q\) holds for infinitely many rational approximations \(p/q\), then it already shows that the real number \(\alpha \in \mathbb{R}\) is irrational, so it is sufficient to prove the irrationality of real numbers.

\begin{lem}[Dirichlet]\label{lem2000.02} 
 Suppose $\alpha \in \mathbb{R}$ is an irrational number. Then there exists an infinite
sequence of rational numbers $p_n/q_n$ satisfying
\begin{equation}
0 < \left|  \alpha -\frac{p_n}{q_n} \right|< \frac{1}{q_n^2}
\end{equation}
for all integers \(n\in \mathbb{N}\).
\end{lem}

\begin{lem} \label{lem2000.03}   Let $\alpha=[a_0,a_1,a_2, \ldots]$ be the continued fraction of a real number, and let $\{p_n/q_n: n \geq 1\}$ be the sequence of convergents. Then
\begin{equation}
0 < \left|  \alpha -\frac{p_n}{q_n} \right|< \frac{1}{a_{n+1}q_n^2}
\end{equation}
for all integers \(n\in \mathbb{N}\).
\end{lem}
This is standard in the literature, the proof appears in \cite[Theorem 171]{HW08}, \cite[Corollary 3.7]{SJ05}, \cite[Theorem 9]{KA97}, and similar references.

\begin{lem} \label{lem2000.05}\hypertarget{lem2000.05}   Let $\alpha=[a_0,a_1,a_2, \ldots]$ be the continued fraction of a real number, and let $\{p_n/q_n: n \geq 1\}$ be the sequence of convergents. Then
\begin{multicols}{2}
 \begin{enumerate} [font=\normalfont, label=(\roman*)]
\item$ \displaystyle 
 \frac{1}{2q_{n+1}q_n} \leq \left | \alpha - \frac{p_n}{q_n}  \right | \leq \frac{1}{q_n^{2}} 
$,
\item$\displaystyle
\frac{1}{2a_{n+1}q_n^2} \leq \left | \alpha - \frac{p_n}{q_n}  \right | \leq \frac{1}{q_n^{2}} 
$,
\end{enumerate}
\end{multicols}
for all integers \(n\in \mathbb{N}\).
\end{lem}
The recursive relation $q_{n+1}=a_{n+1}q_n+q_{n-1}$ links the two inequalities. Confer \cite[Theorem 3.8]{OC63}, \cite[Theorems 9 and 13]{KA97}, et alia. The proof of the best rational approximation stated below, appears in \cite[Theorem 2.1]{RH94}, and \cite[Theorem 3.8]{SJ05}. 
\begin{lem} \label{lem2000.07} \hypertarget{lem2000.07}    Let $\alpha \in \R$ be an irrational real number, and let $\{p_n/q_n: n \geq 1\}$ be the sequence of convergents. Then, for any rational number $p/q \in \Q^{\times}$, 
\begin{multicols}{2}
 \begin{enumerate}[font=\normalfont, label=(\roman*)]
\item$ \displaystyle 
 \left | \alpha q_n - p_n  \right | \leq  \left | \alpha q -p  \right |
$,
\item$\displaystyle
 \left | \alpha - \frac{p_n}{q_n}  \right | \leq  \left | \alpha - \frac{p}{q}  \right |
$,
\end{enumerate}
\end{multicols}

for all sufficiently large \(n\in \mathbb{N}\) such that $q \leq q_n$.
\end{lem}

\begin{thm} {\normalfont(Kronecker approximation theorem)}\label{thm2000.990} Let $\alpha, \beta \in \R^{\times}$ be real numbers, and $\alpha$ irrational. Given a small number $\varepsilon >0$, there exists infinitely many pairs of integers $p,q \in \N$ such that 
\begin{equation}
 \left|  \alpha q - p -\beta \right|< \varepsilon.
\end{equation}
\end{thm}
The $n$th dimensional version and related problems are studied in \cite{AP94}, \cite{GM16}, and similar references.

\subsection{ Irrationalities Measures }
 
The concept of measures of irrationality of real numbers is discussed in \cite[p.\ 556]{WM00}, \cite[Chapter 11]{BB87}, et alii. This concept can be approached from several points of views. 

\begin{dfn} \label{dfn2000.01} {\normalfont The irrationality measure $\mu(\alpha)$ of a real number $\alpha \in \R$ is the infimum of the subset of  real numbers $\mu(\alpha)\geq1$ for which the Diophantine inequality
\begin{equation} \label{eq597.36}
  \left | \alpha-\frac{p}{q} \right | \ll\frac{1}{q^{\mu(\alpha)} }
\end{equation}
has finitely many rational solutions $p$ and $q$. Equivalently, for any arbitrary small number $\varepsilon >0$
\begin{equation} \label{eq597.36b}
  \left | \alpha-\frac{p}{q} \right | \gg\frac{1}{q^{\mu(\alpha)+\varepsilon} }
\end{equation}
for all large $q \geq 1$.
}
\end{dfn}




\end{document}